\documentclass[11pt,letterpaper]{article}

\DeclareSymbolFont{extraup}{U}{zavm}{m}{n}
\DeclareMathSymbol{\varheartsuit}{\mathalpha}{extraup}{86}
\DeclareMathSymbol{\vardiamondsuit}{\mathalpha}{extraup}{87}

\usepackage{fullpage}
\usepackage[utf8]{inputenc} 
\usepackage{hyperref}       
\usepackage{url}            
\usepackage{booktabs}       
\usepackage{amsfonts}       
\usepackage{nicefrac}       
\usepackage{microtype}      
\usepackage{algorithm}
\usepackage{algorithmic}
\usepackage{amsthm}
\usepackage{color}
\usepackage{mathrsfs}
\usepackage{amsmath}
\usepackage{amssymb}
\usepackage{graphics}

\usepackage{natbib}

\usepackage[dvipsnames]{xcolor}
\usepackage{pifont}
\usepackage{mathtools}

\newcommand{\cL}{\mathcal{L}}

\newcommand{\sX}{\mathsf{X}}
\newcommand{\sY}{\mathsf{Y}}
\newcommand{\cO}{\mathcal{O}}

\newcommand{\mW}{\mathbf{W}}

\newcommand{\mL}{\mathbf{K}}
\newcommand{\mI}{\mathbf{I}}
\newcommand{\mP}{\mathbf{P}}
\newcommand{\mQ}{\mathbf{Q}}

\newcommand{\eqdef}{\coloneqq}
\DeclareMathOperator*{\argmin}{arg\,min}
\newcommand{\R}{\mathbb{R}}
\newcommand{\N}{\mathbb{N}}
\newcommand{\Dotprod}[1]{\left\langle#1\right\rangle}
\usepackage{mathtools}
\usepackage{makecell}
\newcommand{\range}{\mathrm{range}}
\newcommand{\bg}{\mathrm{D}}
\newcommand{\Span}{\mathrm{Span}}
\newcommand{\norm}[1]{\|#1\|}
\newcommand{\sqn}[1]{\norm{#1}^2}
\newcommand{\Norm}[1]{\left\|#1\right\|}
\newcommand{\sqN}[1]{\Norm{#1}^2}

\def\<#1,#2>{\langle #1,#2\rangle}

\newtheorem{theorem}{Theorem}
\newtheorem{corollary}{Corollary}

\newtheorem{proposition}{Proposition}

\newtheorem{lemma}{Lemma}

 \author{Adil Salim \qquad Laurent Condat \\ Dmitry Kovalev \qquad Peter Richt\'{a}rik\\
 \phantom{xx}
 \\
 King Abdullah University of Science and Technology (KAUST), \\Thuwal 23955-6900, Kingdom of Saudi Arabia}
 
 \date{Authors' final version. Published in proceedings of \\The 25th International Conference on Artificial Intelligence and Statistics (AISTATS 2022)}

\title{\bf An Optimal Algorithm  for Strongly Convex Minimization under Affine Constraints}
 
\begin{document}
\maketitle

\begin{abstract}%
  Optimization problems under affine constraints appear in various areas of machine learning. We consider the task of minimizing a smooth strongly convex function $F(x)$ under the affine constraint $\mL x = b$, with an oracle providing evaluations of the gradient of $F$ and multiplications by $\mL$ and its transpose. 
We provide lower bounds on the number of gradient computations and matrix multiplications to achieve a given accuracy. Then we propose an accelerated primal--dual algorithm achieving these lower bounds. Our algorithm is the first optimal algorithm for this class of problems.
\end{abstract}

{\small
  \hypersetup{linkcolor=black}
  \tableofcontents
}

\section{Introduction}

We consider the convex optimization problem 
\begin{equation}
	\label{eq:original-pb}
	\min_{x\in \sX}\; F(x) \quad \text{s.t.}\quad \mL x = b,
\end{equation}
where $F$ is a smooth and strongly convex function over $\sX \eqdef \R^d$, $b \in \sY \eqdef \R^p$ is a vector and 
$\mL$ is a nonzero $p \times d$ matrix, 
for some integers $d\geq 1$, $p\geq 1$. We adopt the matrix-vector setting for simplicity of the notations, 
but the formalism holds more generally with arbitrary real Hilbert spaces $\sX$ and $\sY$ and bounded linear operator $\mL: \sX \to \sY$.
We suppose that 
$b$ is in 
the range of $\mL$; then the sought solution to \eqref{eq:original-pb}, denoted by $x^\star$, exists and is unique, by strong convexity.

Problem~\eqref{eq:original-pb} covers a large number of applications in machine learning \citep{sra11,bac12,pol15} and beyond \citep{bau10,sta16,glo16}. Examples include inverse problems in imaging \citep{cha16}, and recovering a model from partial measurements $b$ on the model, in compressed sensing~\citep{gol16} or sketched learning-type applications~\citep{ker18}. In optimal transport, one often looks for measures with fixed marginals, which can be written as an affine equality constraint \citep{pey19}. Network flow optimization takes the form of Problem~\eqref{eq:original-pb}, where $b$ contains the incoming and outgoing rates at source and sink nodes of a network, and $\mL$ is the edge-node incidence matrix \citep{zar13}. Decentralized optimization is a well-known instance of Problem~\eqref{eq:original-pb}, with $\mL$ a gossip matrix (or its square root), and $b=0$~\citep{shi15, scaman2017optimal, gorbunov2019optimal, 
li2020decentralized,li2020optimal,ye2020multi,arjevani2020ideal,kov20,dvi21}. 
If additional affine constraints are added to the decentralized optimization problem, for instance that some elements or linear measurements of the sought model $x^\star$ are fixed, decentralized optimization reverts to Problem~\eqref{eq:original-pb} with nonzero $b$.

For large-scale convex optimization problems like \eqref{eq:original-pb}, primal--dual splitting  algorithms  \citep{bot14,kom15,con19,con20} are well suited, as they are easy to implement and  typically show state-of-the-art performance. The fully-split algorithms do not require the ability to project onto the constraint space $\{x \in \sX :\ \mL x = b\}$, and are therefore particularly adequate in the applications mentioned above. 
Precisely, we say that an  iterative algorithm is fully split if it  produces a sequence of iterates $(x^k)_{k \geq 0} \in \sX^{\N}$ converging to the solution $x^\star$ of \eqref{eq:original-pb}, using only computations of $\nabla F$ and multiplications by $\mL$ and $\mL^T$, the transpose of  $\mL$. 

There exist several fully-split primal--dual algorithms well suited to solve Problem~\eqref{eq:original-pb} and even more general problems~\citep{com12,con13,vu13,yan18,mis19,sal20}. In particular, we can mention the algorithm 
first proposed in \citet{lor11}, and rediscovered independently as the PDFP2O algorithm~\citep{che13} and the Proximal Alternating Predictor-Corrector (PAPC) algorithm \citep{dro15}. For simplicity, we name it the PAPC algorithm. 
When applied to Problem~\eqref{eq:original-pb}, with $F$ strongly convex, the PAPC has been proved to converge linearly by Salim et al.~\citep{sal20}.

In this paper, we focus on the complexity of fully split algorithms to solve Problem~\eqref{eq:original-pb}, which is of primary importance in large-scale applications. That is, we study 
the number of gradient computations and matrix multiplications necessary to reach a given accuracy. We first derive lower bounds for these two quantities. 
No algorithm is known, matching these lower bounds, although nearly optimal algorithms exist in the case $b=0$~\citep{dvi21}. Then, we propose a new accelerated primal--dual algorithm, which matches the lower bounds, and thus is optimal. Our algorithm can be viewed as an accelerated version of the PAPC algorithm.

In summary, our main \textbf{contributions} are the following:
\begin{itemize}
    \item We provide complexity lower bounds for solving Problem~\eqref{eq:original-pb} within the class of algorithms  performing evaluations of $\nabla F$ and multiplications by $\mL$ and $\mL^T$, by a reduction to the technique in ~\citet{scaman2017optimal}.
    \item We propose a new algorithm for solving Problem~\eqref{eq:original-pb}. 
    \item We prove that the complexity of our algorithm matches the lower bounds, and therefore it is optimal.
\end{itemize}
The complexity results presented in this paper are dimension-independent. Under the additional assumption that the dimension is small, one can imagine alternative strategies to solve Problem~\eqref{eq:original-pb} more efficiently. Our algorithm is meant to be applied on high-dimensional problems.

This paper is organized as follows. In Section~\ref{sec:background}, we introduce the notations and assumptions. Then, we summarize our contributions in the light of prior work in Section~\ref{sec:relat}. In Section~\ref{sec:fo}, we define the class of algorithms under study and we derive the corresponding  complexity lower bounds for solving~\eqref{eq:original-pb}. 
Our main algorithm and our main result about its convergence and complexity 
are given in Section~\ref{sec:algo}. Our approach for deriving and analyzing this algorithm is provided in Section~\ref{sec6}. We illustrate our convergence results by numerical experiments in Section~\ref{sec:num}. The technical proofs are postponed to the Supplementary Material.

\section{Mathematical Setting}
\label{sec:background}

Let us make the formulation of the problem \eqref{eq:original-pb} more precise. 
 The convex function $F : \sX \to \R$ is an $L$-smooth and $\mu$-strongly convex function, for some $\mu>0$ and $L>0$; that is, $F$ is differentiable and satisfies the strong convexity inequality
\begin{equation*}
    F(x) + \Dotprod{\nabla F(x),x'-x} + \frac{\mu}{2}\|x-x'\|^2 \leq F(x'), 
\end{equation*}
and the smoothness inequality
\begin{equation*}
    F(x') \leq F(x) + \Dotprod{\nabla F(x),x'-x} + \frac{L}{2}\|x'-x\|^2,
\end{equation*}
for every $(x,x')\in\sX^2$. That is,  $\nabla F$ is $L$-Lipschitz continuous and $F-\frac{\mu}{2}\|\cdot\|^2$ is convex. Moreover, the Bregman divergence of $F$ is denoted by $\bg_F(x,x') \eqdef F(x) - F(x') -\Dotprod{\nabla F(x'),x-x'} \geq 0$.
We have $0 < \mu \leq L$ and we denote by
\begin{equation*}
\kappa \eqdef \frac{L}{\mu} \geq 1
\end{equation*}
the condition number of $F$.

The kernel of the matrix $\mL$ is denoted by $\ker(\mL)$ and its range by $\range(\mL)$. 
We define  the symmetric positive semidefinite matrix $\mW \eqdef \mL^T \mL$. The largest eigenvalue of $\mW$ is denoted by $\lambda_{\max}(\mW)$ and its smallest \textit{positive} eigenvalue by $\lambda_{\min}^{+}(\mW)$. We have $0 < \lambda_{\min}^{+}(\mW) \leq \lambda_{\max}(\mW)$ and we denote by
\begin{equation*}
    \chi(\mW) \eqdef \frac{\lambda_{\max}(\mW)}{\lambda_{\min}^{+}(\mW)} \geq 1
\end{equation*}
the condition number of $\mW$.

The condition number $\kappa$ (resp. $\chi(\mW)$) measures the regularity of $F$ (resp. $\mL$). The complexity results obtained in this paper are (nondecreasing) functions of $\kappa$ (resp. $\chi(\mW)$). 

We can note that $\ker(\mL)=\ker(\mW)$. If $\ker(\mW)=\{0\}$, the solution $x^\star$ to the linear system $\mL x=b$ is unique, so that $F$ does not play any role and Problem~\eqref{eq:original-pb} reverts to solving this linear system. We allow for this case, but it is of course not the focus of this paper.

Finally, we denote by $\iota_{\{b\}}$ the indicator function of $\{b\}$; that is, $\iota_{\{b\}}: y\in \sY \mapsto \{0$ if $y=b$, $+\infty$ otherwise$\}$. 
This function is convex and lower semicontinuous over $\sY$. Denoting by $\partial$ the subdifferential operator~\cite[Section 16]{bau17}, we recall that $\partial \iota_{\{b\}}(y) \neq \emptyset$ if and only if $y = b$.

\section{Related Works and Summary of the Contributions}
\label{sec:relat}

Most algorithms able to solve Problem~\eqref{eq:original-pb} using evaluations of $\nabla F$ and multiplications by $\mL$ and $\mL^T$ can be viewed as primal--dual algorithms. For instance, the Condat-V\~u algorithm~\citep{con13,vu13} and its variants, the PAPC algorithm~\citep{lor11,che13,dro15} 
 can be applied to Problem~\eqref{eq:original-pb}. The PAPC algorithm applied to Problem~\eqref{eq:original-pb} consists in iterating
\begin{equation} 
\label{eq:papc}
\left\lfloor
			\begin{array}{r@{}l}
			x^{k+\frac{1}{2}} &{} \eqdef x^k - \eta \nabla F(x^k) - \eta \mL^T y^k\\
		y^{k+1} &{} \eqdef y^k + \theta (\mL x^{k+\frac{1}{2}} -b)\\
		x^{k+1} &{} \eqdef x^k - \eta \nabla F(x^k) - \eta \mL^T y^{k+1}
			\end{array}\right.
	\end{equation}
for some parameters 	$\eta,\theta > 0$; it converges in general if $\eta \in (0,\frac{2}{L})$ and $\eta \theta \|\mL\|^2\leq 1$, see \citep{con19}, but this particular instance converges linearly~\citep{sal20}, see Table~\ref{tab:rate}. 
 To our knowledge, the first algorithm solving Problem~\eqref{eq:original-pb}, for which linear convergence was proved, has been proposed in \citet{mis19}, see Table~\ref{tab:rate}.

Most of the progress in solving Problem~\eqref{eq:original-pb}, \emph{with $b=0$}, at an accelerated or (nearly) optimal rate have been made recently in the particular case of decentralized optimization~\citep{scaman2017optimal,gorbunov2019optimal,li2020decentralized,li2020optimal,ye2020multi,arjevani2020ideal,kov20,dvi21}. In this case, $\mL$ is typically the square root of a gossip matrix, i.e.\ a symmetric positive semidefinite matrix supported by a graph, whose kernel is the consensus space. In particular, optimal decentralized algorithms have been proposed using acceleration techniques~\citep{nes04,auzinger2011iterative,allen2017katyusha} in \citep{scaman2017optimal,kov20,li2020optimal}. In particular, the algorithm of \citet{scaman2017optimal} relies on the computation of $\nabla F^\ast$, where $F^\ast$ is the Fenchel transform of $F$. Since evaluating $\nabla F^\ast$ is equivalent to minimizing $F$, \citet{kov20} proposed an algorithm relying on $\nabla F$ only. In Machine Learning applications, `full' gradients are often intractable, therefore \citet{li2020optimal} introduced a method relying on stochastic estimates of $\nabla F$ only. Each of these three methods is optimal for the class of algorithms they belong to. 
Our approach can be seen as an extension of~\citep{kov20} to the general setting of linearly constrained minimization, with an arbitrary right hand side $b$.\footnote{We do not assume the knowledge of a solution $\tilde{x}$ to the linear system $\mL x = b$, otherwise one could get back to the case $b = 0$ using a change of variable. One could think of solving this linear system 
approximately 
as a preprocessing step: the typical Conjugate Gradient Method yields $\hat{x}$ with $\|\mL \hat{x}-b\|^2 \leq \epsilon\|b\|^2$  with $\mathcal{O}(\sqrt{\chi}\log(1/\epsilon))$ complexity, where $\chi=\chi(\mW)$.
But, assuming for simplicity that $\|\mL\|=1$, to guarantee that $\|\hat{x}-\tilde{x}\|^2\leq \epsilon\|b\|^2$ for some $\tilde{x}$ with $\mL \tilde{x} = b$,  using the inequality 
$ \|\hat{x}-\tilde{x}\|^2\leq \chi\|\mL \hat{x}-b\|^2 $, the complexity becomes  $\mathcal{O}(\sqrt{\chi}\log(\chi/\epsilon))$. Thus, there is an additional $\log(\chi)$ factor appearing in the complexity, which is not optimal, contrary to the proposed approach. }

In the case where projecting onto the constraint space $\{x \in \sX :\ \mL x = b\}$ is possible, FISTA~\citep{beck2009fast,dos15} is an optimal algorithm for solving Problem~\eqref{eq:original-pb}. FISTA can be seen as Nesterov's acceleration~\citep{nes04} of the classical projected gradient algorithm.

\begin{table*}[t]
  \centering
  \caption{Comparison of the complexity of state-of-the-art algorithms with our results, in terms of gradient computations and matrix multiplications to find  $x\in \sX$ such that $\|x-x^\star\|^2 \leq \varepsilon$. 
  The condition number of $F$ is denoted by $\kappa$ and the condition number of $\mL^T \mL$ is denoted by $\chi$.\medskip} 
  \label{tab:rate}
  \renewcommand{\arraystretch}{1.4} 
  \begin{tabular}[h]{|c|c|c|c|}
    \hline
   Algorithm  & Gradient computations & Matrix multiplications\\
     \hline  
   PAPC algorithm 
   \citep{sal20} & $\cO\left((\kappa+ \chi) \log\frac{1}{\varepsilon}\right)$ & $\cO\left((\kappa+ \chi) \log\frac{1}{\varepsilon}\right)$\\
   \hline
         \citep{mis19} &
 				$\cO\left((\kappa+ \chi) \log\frac{1}{\varepsilon}\right)$ & $\cO\left((\kappa+ \chi) \log\frac{1}{\varepsilon}\right)$ \\
   \hline
   \citep{dvi21} (case $b = 0$) & $\cO\left(\sqrt{\kappa} \log\frac{1}{\varepsilon}\right)$&
 				$\cO\left(\sqrt{\kappa \chi} \log^2\frac{1}{\varepsilon}\right)$ \\
    \hline
   \textbf{Algorithm~\ref{alg:ALV-opt} (This paper, Theorem~\ref{th:main})} &
 				$\boldsymbol{\cO\left(\sqrt{\kappa}\log\frac{1}{\varepsilon}\right)}$&
 				$\boldsymbol{\cO\left(\sqrt{\kappa\chi}\log\frac{1}{\varepsilon}\right)}$\\
     \hline
    \hline
   \textbf{Lower bound (This paper, Theorem~\ref{th:lb})} &
 				$\cO\left(\sqrt{\kappa}\log\frac{1}{\varepsilon}\right)$&
 				$\cO\left(\sqrt{\kappa\chi}\log\frac{1}{\varepsilon}\right)$\\
      \hline
\end{tabular}
\end{table*}

In a nutshell, our approach consists in a rigorous combination of Nesterov's acceleration~\citep{nes04} to minimize a smooth and strongly convex function, and the Chebyshev iteration method \citep{fla50,gol83,auzinger2011iterative,gut02} for linear system solving. 
Our approach allows us to accelerate the PAPC algorithm and, for the first time, to achieve the asymptotic complexity lower bounds.
Our results and the most relevant results of the literature are summarized in Table~\ref{tab:rate}.

\section{First-Order Algorithms for the Problem}
\label{sec:fo}

We now define the family of algorithms considered to solve Problem~\eqref{eq:original-pb}. Informally, this is the family of algorithms using gradient computations and matrix multiplications. Since no particular structure is assumed on $\mL$, any multiplication of the iterates by $\mL$ must be followed by a multiplication by $\mL^T$ in order to map the iterates back into the optimization space $\sX$, before an application of $\nabla F$. Hence, we consider the wide class of Black-Box First Order algorithms using $\nabla F$, $\mL$ and $\mL^T$, denoted by BBFO$(\nabla F,\mL)$, which generate a sequence of vectors $(x^n)_{n\in\mathbb{N}}\in\sX^\mathbb{N}$ such that
\begin{align*}
	x^{n+1} \in {}&\Span\Big(x^0,\ldots,x^n,\nabla F(x^0),\ldots, \nabla F(x^n), \\
	&\quad\mL^T \Span\big(b, \mL x^0,\ldots,\mL x^n,\mL \nabla F(x^0),\ldots, \\
	&\quad\mL \nabla F(x^n)\big)\Big)
\end{align*}
and do not apply the operators $\nabla F$, $\mL$ and $\mL^T$ to other vectors.
It is important to note that the index $n$ need not coincide with the iteration counter of an iterative algorithm: each $x^{n}$ can correspond to an intermediate vector in $\sX$ obtained after any computation or sequence of computations during the course of the algorithm.

\begin{theorem}[Lower bounds]
\label{th:lb}
	Let $\chi \geq 1$. There exist a vector $b_0$, a matrix $\mL_0$ such that the condition number of $\mL_0^T \mL_0$ is $\chi$, and a smooth and strongly convex function $F_0$ with condition number $\kappa$, such that the following holds: for any $\varepsilon > 0$, any BBFO$(\nabla F_0,\mL_0)$ algorithm requires at least\\[-5mm]
		\begin{itemize}
			\item $\Omega(\sqrt{\kappa \chi} \log(1/\varepsilon))$ multiplications by $\mL_0$,\\[-5mm]
			\item $\Omega(\sqrt{\kappa \chi} \log(1/\varepsilon))$ multiplications by $\mL_0^T$,\\[-5mm]
			\item $\Omega(\sqrt{\kappa} \log(1/\varepsilon))$ computations of $\nabla F_0$,
		\end{itemize}
		to output a vector $x$ such that $\sqn{x-x^{\star}} < \varepsilon$, where $x^{\star} = \argmin_{\{x\ :\ \mL_0 x = b_0\}} F_0(x).$
\end{theorem}
Theorem~\ref{th:lb} provides lower bounds on the number of gradient computations and matrix multiplications needed to reach $\varepsilon$ accuracy, which here means that $\sqn{x-x^{\star}} \leq \varepsilon$.

\begin{proof} We follow the ideas of~\citet{scaman2017optimal}, in the context of decentralized optimization, to exhibit  worst-case function $F_0$ and matrix $\mL_0$. 
Let $\chi \geq 1$.
	\paragraph{``Bad'' function $F_0$ and ``bad'' matrix $\mL_0$.} 
	Consider the family of smooth and strongly convex functions $(f_i)_{i=1}^n$  
	and the matrix $\mW$ with condition number $\chi$ given by~\cite[Corollary 2]{scaman2017optimal}. Denote by $\kappa$ the common condition number of $f_i$. Set $F_0(x_1,\ldots,x_n) \eqdef \sum_{i=1}^n f_i(x_i)$, $\mL_0 \eqdef \sqrt{\mW}$ and $b_0 \eqdef 0$.
	Then, the condition number of $F$ is $\kappa$ and the condition number of $\mW = \mL_0^T \mL_0$ is $\chi$. Moreover, $\mW$ is a gossip matrix~\cite[Section 2.2]{scaman2017optimal}.

	\paragraph{BBFO$(\nabla F_0,\mL_0)$ are decentralized optimization algorithms.}
	Any BBFO algorithm using these operators $\nabla F_0$, $\mL_0$, $\mL_0^T$ can be rewritten as a function of $\nabla F_0$ and $\mW = \mL_0^T \mL_0$. Indeed,  
	\begin{align*}
		&\Span\Big(x^0,\ldots,x^n,\nabla F_0(x^0),\ldots, \nabla F_0(x^n), \\
		&\quad\mL_0^T \Span\big(b_0,\mL_0 x^0,\ldots,\mL_0 x^n,\mL_0 \nabla F_0(x^0),\ldots, \\
		&\quad\mL_0 \nabla F_0(x^n)\big)\Big)\\
		={}&\Span\Big(x^0,\ldots,x^n,\nabla F_0(x^0),\ldots, \nabla F_0(x^n), \\
		&\quad\Span\big(\mW x^0,\ldots,\mW x^n,\mW \nabla F_0(x^0),\ldots, \\
		&\quad\mW \nabla F_0(x^n)\big)\Big).
	\end{align*}
Since $\mW$ is a gossip matrix, BBFO$(\nabla F_0,\mL_0)$ algorithms are therefore Black-box optimization procedures using $\mW$, in the sense of~\cite[Section 3.1]{scaman2017optimal}. In other words, BBFO$(\nabla F_0,\mL_0)$ algorithms are decentralized optimization algorithms over a network, in which communication amounts to multiplication by $\mW$, and local computations correspond to evaluations of $\nabla F$. 
	\paragraph{Any solution to~\eqref{eq:original-pb} is a solution to a decentralized optimization problem.}
	Since $\ker(\mW)$ is the consensus space, $x^{\star} = \argmin_{\{x\ :\ \mW x = 0\}} F_0(x)$ can be written as $x^{\star} = (x_0^{\star},\ldots,x_0^{\star})$ where $x_0^{\star} = \argmin \frac{1}{n}\sum_{i=1}^n f_i$.
	\paragraph{BBFO$(\nabla F_0,\mL_0)$ algorithms cannot outperform the lower bounds of decentralized algorithms.}
	As shown in \citet[Corollary 2]{scaman2017optimal}, for any $\varepsilon >0$, any Black-box optimization procedure using $\mW$ requires at least $\Omega\left(\sqrt{\kappa \chi}\log(1/\varepsilon)\right)$ communication rounds, and at least $\Omega\left(\sqrt{\kappa}\log(1/\varepsilon)\right)$ gradient computations to output $x = (x_1, \ldots, x_n)$ such that $\sqn{x - x^{\star}} < \varepsilon$, where $x^{\star} = \argmin F_0.$ \textit{In particular}, for any $\varepsilon >0$, any BBFO$(\nabla F_0,\mL_0)$ algorithm requires at least $\Omega\left(\sqrt{\kappa \chi}\log(1/\varepsilon)\right)$ multiplications by $\mL_0^T \mL_0$, and at least $\Omega\left(\sqrt{\kappa}\log(1/\varepsilon)\right)$ computations of $\nabla F_0$ to output $x = (x_1, \ldots, x_n)$ such that $\sqn{x - x^{\star}} < \varepsilon$, where $x^{\star} = \argmin F_0.$ 
	
	Finally, one multiplication by $\mW$ is equivalent to one multiplication by $\mL_0$ followed by one multiplication by $\mL^T_0$.
\end{proof}

\section{Proposed Algorithm}
\label{sec:algo}

\begin{figure*}[t!]
\begin{minipage}{.482\textwidth}
\begin{algorithm}[H]
	\caption{Proposed algorithm}
	\label{alg:ALV-opt}
	\begin{algorithmic}[1]
		\STATE {\bf Parameters:}  $x^0 \in \sX$, $N \in \mathbb N^{*}$, $\tau \in (0,1)$,
		\STATE $\lambda_1, \lambda_2,\eta,\theta ,\alpha>0$
		\STATE $x_f^0 \eqdef x^0$, $u^0 \eqdef 0_{\sX}$
		\FOR{$k=0,1,\ldots$}{}
		\STATE $x_g^k \eqdef \tau x^k + (1-\tau)x_f^k$\label{alg:ALV-opt:line:x:1}
		\STATE $x^{k+\frac{1}{2}} \eqdef (1+\eta\alpha)^{-1}\big(x^k - \eta (\nabla F(x_g^k) $
		\STATE $\ \ {}- \alpha  x_g^k + u^k)\big)$\label{alg:ALV-opt:line:x:2}
		\STATE $r^k \eqdef  \theta \big(x^{k+\frac{1}{2}}$
		\STATE $\ \ {}-\mathrm{Chebyshev}(x^{k+\frac{1}{2}},\mL,b,N,\lambda_1,\lambda_2)\big)$
		\STATE $u^{k+1} \eqdef u^k + r^k$
		\STATE $x^{k+1} \eqdef x^{k+\frac{1}{2}} -\eta (1+\eta\alpha)^{-1}r^k$
		\STATE $x_f^{k+1} \eqdef x_g^k + \tfrac{2\tau}{2-\tau}(x^{k+1} - x^k)$\label{alg:ALV-opt:line:x:4}
		\ENDFOR
	\end{algorithmic}
\end{algorithm}
	\end{minipage}
	\ \ \ \ \ \ \begin{minipage}{.482\textwidth}
\begin{algorithm}[H]
	\caption{Chebyshev iteration} 
	\label{alg:M9}
	\begin{algorithmic}[1]
		\STATE {\bf Parameters:} $z^0\in \sX, \mL, 
		b \in \sY, N \in \mathbb N^{*}, \lambda_1>0$, $\lambda_2 > 0.$
		\STATE $\rho\eqdef\big(\lambda_{1}-\lambda_{2}\big)^2/16$, $\nu\eqdef(\lambda_{1}+\lambda_{2})/2$
		\STATE $\gamma^0 \eqdef -\nu/2$
		\STATE $p^0 \eqdef -\mL^T (\mL z^0-b)/\nu$
		\STATE $z^1 \eqdef z^0 + p^0$
		\FOR{$i=1,\ldots,N-1$}{}
		\STATE $\beta^{i-1}\eqdef\rho/\gamma^{i-1}$
		\STATE $\gamma^i\eqdef-(\nu+\beta^{i-1})$
		\STATE $p^i\eqdef\big(\mL^T (\mL z^{i}-b)+\beta^{i-1}p^{i-1}\big)/\gamma^i$
		\STATE $z^{i+1}\eqdef z^i+p^i$
		\ENDFOR
	\STATE {\bf Output:} $z^N$
	\end{algorithmic}
\end{algorithm}
\end{minipage}
\end{figure*}

In this section, we present our main algorithm, Algorithm~\ref{alg:ALV-opt}, and our main convergence result, Theorem~\ref{th:main}. The derivations and proofs are deferred to Section~\ref{sec6}. Algorithm~\ref{alg:M9} implements the classical Chebyshev iteration~\citep{fla50,gol83,auzinger2011iterative,gut02}, see Section~\ref{sec62} for details. It is used as a subroutine in Algorithm~\ref{alg:ALV-opt} and denoted by $\mathrm{Chebyshev}$, with its parameters passed as arguments. We stress here that the Chebyshev iteration is diverted from its usual use, which is solving linear systems, and is used here as a preconditioner (a similar idea appears in~\citet{bredies2015preconditioned}). Although Algorithm~\ref{alg:ALV-opt} runs  at every iteration a number $N$ of Chebyshev iterations, there is no approximation or truncation error here: Algorithm~\ref{alg:ALV-opt} converges to the exact solution $x^\star$ of Problem~\eqref{eq:original-pb}. This is achieved without solving the full linear system $\mL x=b$ at each iteration.

\begin{theorem}[Convergence of Algorithm~\ref{alg:ALV-opt}]
		\label{th:main}
		Consider $\lambda_1 \geq \lambda_{\max}(\mW)$ and $\lambda_2$ such that $0 < \lambda_2 \leq \lambda_{\min}^+(\mW)$. Let $\chi \eqdef \frac{\lambda_1}{\lambda_2}$ and choose $N \geq \sqrt{\chi}$.
		
		Set the parameters $\tau, \eta, \theta, \alpha$ as
		$\tau \eqdef \min \left\{1,\frac{1}{2}\sqrt{\frac{19}{15 \kappa} }\right\}
		$, $
			\eta \eqdef \frac{1} {4\tau L}$, $\theta \eqdef \frac{15}{19 \eta}$, and $\alpha \eqdef \mu$.
Then, there exists $C \geq 0$ such that
		\begin{align*}
			& \frac{1}{\eta}\sqN{x^k - x^\star} + 
			\frac{2(1-\tau)}{\tau}\bg_F(x_f^k,x^{\star})\\
			&\leq
			\left(1 + \frac{1}{4}\min\left\{\frac{15}{19},\sqrt{\frac{15}{19 \kappa}}\right\}\right)^{-k}
			C.
		\end{align*}
		Moreover, for every $\varepsilon >0$, Algorithm~\ref{alg:ALV-opt} finds $x^k$ for which $\sqn{x^k - x^\star} \leq \varepsilon$ using
	 $\cO\left(\sqrt{\kappa}\log(1/\varepsilon)\right)$  gradient computations and 
		$\cO\left(N\sqrt{\kappa}\log(1/\varepsilon)\right)$ matrix multiplications with $\mL$ or $\mL^T$. 
	\end{theorem}\smallskip
	
\begin{corollary}[Tight version of Theorem~\ref{th:main}]
\label{cor:main}
Set the parameters $\lambda_1,\lambda_2,N,\tau, \eta, \theta, \alpha$ to $\lambda_1 = \lambda_{\max}(\mW)$, $\lambda_2 = \lambda_{\min}^+(\mW)$, $N = \lceil \sqrt{\chi(\mW)} \rceil$,
		$\tau = \min \left\{1,\frac{1}{2}\sqrt{\frac{19}{15 \kappa} }\right\}
		$, $
			\eta = \frac{1} {4\tau L}$, $\theta = \frac{15}{19 \eta}$, and $\alpha = \mu$. Then, for every $\varepsilon >0$, Algorithm~\ref{alg:ALV-opt} finds $x^k$ for which $\sqn{x^k - x^\star} \leq \varepsilon$ using
	 $\cO\left(\sqrt{\kappa}\log(1/\varepsilon)\right)$  gradient computations and 
		$\cO\left(\sqrt{\kappa \chi(\mW)}\log(1/\varepsilon)\right)$ matrix multiplications with $\mL$ or $\mL^T$. 
\end{corollary}
The complexity result given by Corollary~\ref{cor:main} is summarized in Table~\ref{tab:rate}.  Algorithm~\ref{alg:ALV-opt} is a BBFO algorithm because each step of each iteration is a BBFO update.  Thus, the complexity of Algorithm~\ref{alg:ALV-opt} matches the lower bounds of Theorem~\ref{th:lb}, in terms of both gradient computations and matrix multiplications.

\section{Derivation of the Algorithm and Proof of Theorem~\ref{th:main}}
\label{sec6}
In this section, we explain how we derive our main algorithm from the PAPC algorithm and prove Theorem~\ref{th:main} step by step.
First, we derive the primal--dual optimality conditions associated to Problem~\eqref{eq:original-pb}.

\subsection{Primal--Dual Optimality Conditions}
First, note that $\argmin_{\mL x = b} F(x) = \argmin F(x) + \iota_{\{b\}}(\mL x)$. Define the strongly convex function $G : x \mapsto F(x) + \iota_{\{b\}}(\mL x)$. Then, $0 \in \partial G(x^{\star}) = \nabla F(x^{\star}) + \mL^T \partial \iota_{\{b\}}(\mL x^{\star})$~\cite[Theorem 16.47]{bau17}. This means that there exists $y^{\star} \in \partial \iota_{\{b\}}(\mL x^{\star})$ such that $0 = \nabla F(x^{\star}) + \mL^T y^{\star}$. Besides, $ \partial \iota_{\{b\}}(\mL x^{\star})$ is nonempty if and only if $\mL x^{\star} = b$. Finally, the pair $(x^{\star},y^{\star})$ must satisfy 
\begin{equation}
\label{eq:pd1}
\left\{\begin{array}{l}
    0 = \nabla F(x^{\star}) + \mL^T y^{\star},\\
    0 = -\mL x^{\star} + b.
\end{array}\right. 
\end{equation}
These equations are called primal--dual optimality conditions, and are also the first-order conditions associated to the Lagrangian function 
   $ \cL(x,y) \eqdef F(x) + \Dotprod{\mL x - b,y}$
associated to Problem~\eqref{eq:original-pb}. Moreover, $(x^{\star},y^{\star})$ is called an optimal primal--dual pair. If $(x^{\star},y^{\star})$ is an optimal primal--dual pair, then $(x^{\star},y^{\star} + \bar{y})$, where $\bar{y} \in \ker(\mL^T)$, is also an optimal primal--dual pair.
Thus, in the sequel, we denote by $(x^{\star},y^{\star})$ the only optimal primal--dual pair such that $y^{\star} \in \range(\mL)$; that is, such that 
\begin{equation}
\left\{\begin{array}{l}
    0 = \nabla F(x^{\star}) + \mL^T y^{\star}, \quad y^{\star} \in \range(\mL), \\
    0 = -\mL x^{\star} + b.
    \end{array}\right. \label{eq:pd}
\end{equation}
We can note that the sequence of iterates $(x^k,y^k)$ of the PAPC algorithm, shown in \eqref{eq:papc}, converges linearly to  $(x^\star,y^\star)$~\cite[Theorem 8]{sal20}, as reported in 
 Table~\ref{tab:rate}.

\subsection{Nesterov's Acceleration}
The first step to derive Algorithm~\ref{alg:ALV-opt} is to propose a variant of the PAPC~\eqref{eq:papc} using Nesterov's acceleration~\citep{nes04}. Nesterov acceleration is now
classical for proximal gradient descent but its extension to primal-dual settings remains an open area. 
This intermediate algorithm is Algorithm~\ref{alg:ALV}, shown above. Its convergence is stated in Proposition~\ref{th:ALV}. 

\begin{algorithm}[t]
    \caption{Intermediate algorithm}
	\label{alg:ALV}
	\begin{algorithmic}[1]
		\STATE {\bf Parameters:}  $x^0 \in \sX$,  $y^0 = 0_{\sY}$, 
		$\eta,\theta,\alpha>0$, $\tau \in (0,1)$
		\STATE Set $x_f^0 = x^0$
		\FOR{$k=0,1,2,\ldots$}{}
		\STATE $x_g^k \eqdef \tau x^k +  (1-\tau)x_f^k$\label{alg:ALV:line:x:1}
		\STATE $x^{k+\frac{1}{2}} \eqdef (1+\eta\alpha)^{-1}(x^k - \eta (\nabla F(x_g^k) - \alpha  x_g^k + \mL^T y^k))$\label{alg:ALV:line:x:2}
		\STATE $y^{k+1} \eqdef y^k  + \theta (\mL x^{k+\frac{1}{2}} - b)$ \label{alg:ALV:line:y}
		\STATE $x^{k+1} \eqdef (1+\eta\alpha)^{-1}(x^k - \eta (\nabla F(x_g^k) - \alpha  x_g^k + \mL^T y^{k+1}))$\label{alg:ALV:line:x:3}
		\STATE $x_f^{k+1} \eqdef x_g^k + \tfrac{2\tau}{2-\tau}(x^{k+1} - x^k)$\label{alg:ALV:line:x:4}
		\ENDFOR
	\end{algorithmic}
\end{algorithm}
\begin{proposition}[Algorithm~\ref{alg:ALV}]
		\label{th:ALV}
		Consider $\lambda_1 \geq \lambda_{\max}(\mW)$ and $\lambda_2$ such that $0 < \lambda_2 \leq \lambda_{\min}^+(\mW)$. Denote $\chi \eqdef \frac{\lambda_1}{\lambda_2}$.
		
		Set the parameters of Algorithm~\ref{alg:ALV} as $\tau \eqdef \min \left\{1,\frac{1}{2}\sqrt{\frac{\chi}{\kappa} }\right\}
		$, $
			\eta \eqdef \frac{1} {4\tau L}$, $\theta \eqdef \frac{1}{\eta\lambda_1}$, and $\alpha \eqdef \mu$.
		Then,
		\begin{align}
		&\frac{1}{\eta}\sqN{x^k - x^{\star}} + \frac{\eta\alpha}{\theta(1+\eta\alpha)}\sqN{y^k - y^{\star}} \\
		&+
			\frac{2(1-\tau)}{\tau}\bg_F(x_f^k, x^{\star})
			\leq
			\left(\!1 + \frac{1}{4}\min\left\{\frac{1}{\sqrt{\kappa\chi}},\frac{1}{\chi}\right\}\!\right)^{-k}
			C,\notag
		\end{align}
		where
		$C \eqdef \frac{1}{\eta}\sqN{x^0 - x^{\star}} + \frac{1}{\theta}\sqn{y^0 - y^{\star}} +
		\frac{2(1-\tau)}{\tau}\bg_F(x_f^0, x^{\star}).
	$
\end{proposition}
Proposition~\ref{th:ALV} states the linear convergence of the distance between the iterates and the primal--dual optimal point. In particular, if $\lambda_1 = \lambda_{\max}(\mW)$ and $\lambda_2 = \lambda_{\min}^+(\mW)$, then $\|x-x^\star\|^2 \leq \varepsilon$ after
\begin{equation*}\cO\left(\left(\sqrt{\kappa \chi(\mW)} + \chi(\mW)\right)\log\left({\textstyle\frac{1}{\varepsilon}}\right)\right)
\end{equation*} 
gradient computations and matrix multiplications. Besides, Proposition~\ref{th:ALV} states the linear convergence of the Bregman divergence of $F$. Using~\eqref{eq:pd}, one can check that the Bregman divergence of $F$ is equal to the restricted primal--dual gap, in particular: $\bg_F(x_f^k,x^\star) = \cL(x_f^k,y^\star) - \cL(x^\star,y^k)$.

The proof of Proposition~\ref{th:ALV} is provided in the Supplementary Material. The main tool of the proof is the following representation of Algorithm~\ref{alg:ALV}.

We denote by $\mQ$  
the $(d+p)\times(d+p)$ matrix defined blockwise by
\begin{equation}
\label{eq:P}
		\mQ \eqdef \begin{bmatrix}
		\frac{1}{\eta}\mI_\sX & 0\\
		0&\frac{1}{\theta}\mI_\sY - \frac{\eta}{1+\eta\alpha}\mL \mL^T
		\end{bmatrix},
	\end{equation}
where $\mI_\sX$ (resp. $\mI_{\sY}$) is the identity matrix over $\sX$ (resp. $\sY$). 
\begin{lemma}
\label{lem:FB}
The following equality holds:
	\begin{equation}\label{ALV:eq:xy1}
		\mQ \begin{bmatrix}
		x^{k+1} - x^k\\y^{k+1} - y^k
		\end{bmatrix}
		\!=\!
		\begin{bmatrix}
			\alpha (x_g^k - x^{k+1}) - (\nabla F(x_g^k) + \mL^T y^{k+1})\\
			\mL x^{k+1} - b
		\end{bmatrix}.
	\end{equation}
    
\end{lemma}
Lemma~\ref{lem:FB}, proved in the Supplementary Material, enables to view Algorithm~\ref{alg:ALV} as a variant of the Forward--Backward algorithm involving monotone operators, see~\cite[Section 26.14]{bau17} or~\citep{con19} for more details. The Forward--Backward algorithm is a fixed-point algorithm. For instance, one can see in Equation~\eqref{ALV:eq:xy1} that a fixed point $(x^k,y^k) = (x^{\star},y^{\star})$ is a solution to~\eqref{eq:pd1}. Hence, Algorithm~\ref{alg:ALV} can be viewed as an accelerated primal--dual fixed-point algorithm.

\subsection{Chebyshev's Acceleration}\label{sec62}
Our main Algorithm~\ref{alg:ALV-opt} is obtained as a particular instantiation of Algorithm~\ref{alg:ALV}. More precisely, we use a finite number of steps of the Chebyshev iteration~\citep{fla50,gol83,auzinger2011iterative,gut02} to precondition the linear system and accelerate the resolution of Problem~\eqref{eq:original-pb}. This idea was already applied in  the particular setting of decentralized optimization~\citep{scaman2017optimal,scaman2018optimal}.

Consider a polynomial $\mP$ 
such that, for every eigenvalue $t$ of $\mW$, $\mP(t) \geq 0$ and ($\mP(t) = 0 \Leftrightarrow t = 0$).
Since $\mL x^{\star} = b$, 
\begin{align*}
    \mL x = b
    &\Leftrightarrow\  \mL (x-x^{\star}) = 0
    \Leftrightarrow\ \mL^T \mL (x-x^{\star}) = 0\\
    &
    \Leftrightarrow\ \mW (x-x^{\star}) = 0
    \Leftrightarrow\ \mP(\mW) (x-x^{\star}) = 0\\
&    \Leftrightarrow\ \sqrt{\mP(\mW)} (x-x^{\star}) = 0\\
&    \Leftrightarrow\ \sqrt{\mP(\mW)} x = \sqrt{\mP(\mW)} x^{\star}.
\end{align*}
Therefore, the problem
\begin{equation}
	\label{eq:equiv-pb}
	\min_{x \in \sX}\  F(x) \quad \text{s.t.}\quad \sqrt{\mP(\mW)} x = \sqrt{\mP(\mW)} x^{\star},
\end{equation}
is equivalent to Problem~\eqref{eq:original-pb}. 
Consequently, to solve Problem~\eqref{eq:original-pb}, one can apply Algorithm~\ref{alg:ALV} by replacing $\mL$ by $\sqrt{\mP(\mW)}$ and $b$ by $\sqrt{\mP(\mW)} x^{\star}$; we will see below that $x^\star$ is not needed in the computations, only $b$ is.  Since $\sqrt{\mP(\mW)}$ is symmetric, this leads to the following algorithm:
\begin{equation}
   \left\lfloor\begin{array}{r@{\,}l}
    x_g^k &\eqdef \tau x^k +  (1-\tau)x_f^k\\
		x^{k+\frac{1}{2}} &\eqdef (1+\eta\alpha)^{-1}\big(x^k - \eta (\nabla F(x_g^k) - \alpha  x_g^k \\
		&\quad+ \sqrt{\mP(\mW)} y^k)\big)\\
		y^{k+1} &\eqdef y^k  + \theta \big(\sqrt{\mP(\mW)} x^{k+\frac{1}{2}} - \sqrt{\mP(\mW)} x^{\star}\big)\\
		x^{k+1} &\eqdef (1+\eta\alpha)^{-1}\big(x^k - \eta (\nabla F(x_g^k) - \alpha  x_g^k \\
		&\quad{}+ \sqrt{\mP(\mW)} y^{k+1})\big)\\
		x_f^{k+1} &\eqdef x_g^k + \tfrac{2\tau}{2-\tau}(x^{k+1} - x^k)
		\end{array}\right. .
\end{equation}
After applying the change of variable $u^k \eqdef 
\sqrt{\mP(\mW)} y^k$, we get:
\begin{equation}
\label{eq:algo-equiv}
\left\lfloor\begin{array}{r@{\,}l}
    x_g^k &\eqdef \tau x^k +  (1-\tau)x_f^k\\
		x^{k+\frac{1}{2}} &\eqdef (1+\eta\alpha)^{-1}\big(x^k - \eta (\nabla F(x_g^k) - \alpha  x_g^k + u^k)\big)\\
		u^{k+1} &\eqdef u^k  + \theta \big({\mP(\mW)} x^{k+\frac{1}{2}} - {\mP(\mW)} x^{\star}\big)\\
		x^{k+1} &\eqdef (1+\eta\alpha)^{-1}\big(x^k - \eta (\nabla F(x_g^k) - \alpha  x_g^k + u^{k+1})\big)\\
		x_f^{k+1} &\eqdef x_g^k + \tfrac{2\tau}{2-\tau}(x^{k+1} - x^k).\end{array}\right. 
\end{equation}

To obtain Algorithm~\ref{alg:ALV-opt} and Theorem~\ref{th:main}, we have to choose a suitable polynomial $\mP$ and show how to compute ${\mP(\mW)} x^{k+\frac{1}{2}} - {\mP(\mW)} x^{\star}$ efficiently. 

\subsubsection{Choice of $\mP$}
\label{sec:chiP}
The goal is to make Problem~\eqref{eq:equiv-pb} better conditioned than Problem~\eqref{eq:original-pb}. For this, we want $\mP$ to cluster all the positive eigenvalues of $\mW$ around the same value, say 1 (the scaling of $\mP$ does not matter, since it is compensated by the stepsizes). To that aim, the best choice is to set $\mP$ as 1 minus a Chebyshev polynomial of appropriate degree~\cite[Theorem 6.1]{auzinger2011iterative}. More precisely, let $\boldsymbol{T}_n$ be the Chebyshev polynomial of the first kind of degree $n\geq 0$, which is such that $\{ \boldsymbol{T}_n(t)\ :\ t\in [-1,1] \} = [-1,1]$. Let $\lambda_1 \geq \lambda_{\max}(\mW)$ and 
$0<\lambda_2 \leq \lambda_{\min}^+(\mW)$ be upper and lower bounds of the eigenvalues of $\mW$. Set $\chi \eqdef  \lambda_1/\lambda_2\geq \chi(\mW) \geq 1$. 

 If $\lambda_1=\lambda_2$, no preconditioning is necessary and we could just set $\mP(\mW)=\mW$. 
So, let us assume that $\lambda_2<\lambda_1$ (the derivations can be shown to be still valid if  $\lambda_2=\lambda_1$).

For every $n\geq 1$, we define the shifted Chebyshev polynomial $\widetilde{\boldsymbol{T}}_n$ as
\begin{equation}
\widetilde{\boldsymbol{T}}_n(t)=\frac{\boldsymbol{T}_n\big((\lambda_1+\lambda_2-2t)/(\lambda_1-\lambda_2)\big)}{\boldsymbol{T}_n\big((\lambda_1+\lambda_2)/(\lambda_1-\lambda_2)\big)}.
\end{equation}
Then, for every $n \geq 1$, $\widetilde{\boldsymbol{T}}_n(0)=1$, 
$\widetilde{\boldsymbol{T}}_n(t)$ decreases monotonically for $t\in [0,\lambda_2]$, and
\begin{align}
\max_{t\in [\lambda_2,\lambda_1]} |\widetilde{\boldsymbol{T}}_n(t)| &= \frac{1}{\boldsymbol{T}_n\big((\lambda_1+\lambda_2)/(\lambda_1-\lambda_2)\big)} \\
=&  \frac{2\zeta^n}{1+\zeta^{2n}}<1,\ \mbox{where}\  \zeta=\frac{\sqrt{\chi}-1}{\sqrt{\chi}+1}<1,\notag
\end{align}
see~\cite[Corollary 6.1]{auzinger2011iterative}.
Hence, if $N \geq \sqrt{\chi}$, then 
\begin{equation}
\label{eq:cheb}
\max_{t\in [\lambda_2,\lambda_1]} |\widetilde{\boldsymbol{T}}_N(t)| < 0.266 <\frac{4}{15}.
\end{equation}
Indeed,
$-1/\ln((t-1)/(t+1))<t/2$ for every $t\geq 1$, therefore by setting $t = \sqrt{\chi}$ we obtain $N\geq \sqrt{\chi} \Rightarrow N > -2/\ln(\zeta) \Rightarrow \zeta^N < e^{-2} \Rightarrow 2\zeta^N/(1+\zeta^{2N})<0.266$.  

Therefore, we set
\begin{equation}
\label{eq:defP}
\mP \eqdef 1-\widetilde{\boldsymbol{T}}_N
\end{equation}
for some $N\geq \sqrt{\chi}$. Then, we have 
\begin{align*}
\label{eq:chiP}
\lambda_{\max}(\mP(\mW)) &\leq \!\max_{t\in [\lambda_2,\lambda_1]} \!\mP(t) \leq 1+\!\!\max_{t\in [\lambda_2,\lambda_1]} |\widetilde{\boldsymbol{T}}_N(t)| \leq \frac{19}{15},\\
\lambda_{\min}^+(\mP(\mW)) &\geq \!\min_{t\in [\lambda_2,\lambda_1]}\! \mP(t) \geq 1 -\!\!\max_{t\in [\lambda_2,\lambda_1]} |\widetilde{\boldsymbol{T}}_N(t)| \geq \frac{11}{15},\\
\chi\big(\mP(\mW)\big)&\leq \frac{19}{11}.
\end{align*}

\subsubsection{Efficient Computation of ${\mP(\mW)}x- {\mP(\mW)} x^{\star}$ without Knowing $x^\star$}
\label{sec:algo-equiv-P}
We still have to show how to compute ${\mP(\mW)}x- {\mP(\mW)} x^{\star}$, for any $x\in \sX$. 
Consider $N \geq 1$ and $\mP$ defined in~\eqref{eq:defP}. Now, we can observe that Algorithm~\ref{alg:ALV-opt} is equivalent to the iterations~\eqref{eq:algo-equiv}, and there remains to prove that for every $x \in \sX$,
\begin{equation}
    \mP(\mW)x - \mP(\mW)x^\star = x - \mathrm{Chebyshev}(x,\mL,b,N,\lambda_1,\lambda_2).\label{eqch11}
\end{equation}
The vector $z^N = \mathrm{Chebyshev}(x,\mL,b,N)$ is the $N^{\text{th}}$ iterate of the classical Chebyshev iteration to solve the linear system $\mL z= b$, or equivalently $\mW z = \mL^T b$, starting with some initial guess $z^0=x$, using the recurrence relation of the Chebyshev polynomials $\widetilde{\boldsymbol{T}}_N$, see Algorithm 4 in \citet{gut02}\footnote{Several recurrence relations  can be used to compute $\widetilde{\boldsymbol{T}}_n$, and we chose Algorithm 4 in \citet{gut02} because it is proved to be numerically stable.}. The rest of the proof is given in the Supplementary Material.

\section{Experiments}
\label{sec:num}
\begin{figure*}[t]
\centering
$\!\!\!\!$\includegraphics[scale=0.80]{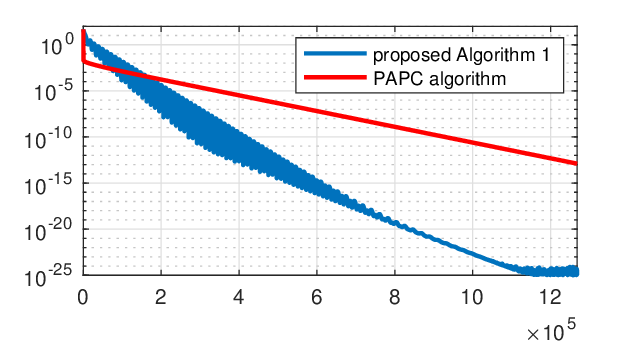}$\!\!$\raisebox{10pt}{\includegraphics[scale=0.80]{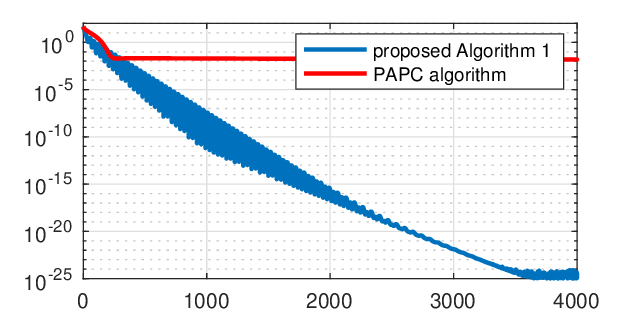}}
\caption{Error $\|x-x^\star\|^2$ with respect to the number of calls to $\mL$ and $\mL^T$ to obtain $x$ (left) and to the number of calls to $\nabla F$, equal to the number $k$ of iterations, to obtain $x=x^k$  (right).}
\label{fig1}
\end{figure*}

We illustrate the performance of our Algorithm~\ref{alg:ALV-opt} in a compressed-sensing-type experiment: we want to estimate a sparse vector $x^\sharp\in \sX=\mathbb{R}^d$, with $d=1000$, having $50$ randomly chosen nonzero elements (equal to 1) from $b=\mL x^\sharp \in \sY=\mathbb{R}^p$, with $p=250$, where $\mL$ has random i.i.d. Gaussian elements and its nonzero singular values are modified so that they span the interval $[1/\sqrt{\chi},1]$ for some prescribed value of $\chi$.
Solving Problem~\eqref{eq:original-pb} with $F$ the $\ell_1$ norm yields perfect reconstruction with $x^\star=x^\sharp$. Thus, without pretending in any way that this is the best way to solve this estimation problem, we solve Problem~\eqref{eq:original-pb} with $F$ a $L$-smooth and $\mu$-strongly convex approximation of the $\ell_1$ norm: we set $F:x=(x_i)_{i=1}^d \in \sX \mapsto \sum_{i=1}^d f(x_i)$ with $f:t\in\mathbb{R}\mapsto \sqrt{t^2+e^2}+(e/2) t^2$, for some $e>0$, so that $L=1/e+e$, $\mu=e$, $\kappa=L/\mu=1+1/e^2$. So, given a prescribed value of $\kappa$, we set $e = \sqrt{1/(\kappa-1)}$. The results are shown in Figure~\ref{fig1} for 
Algorithm~\ref{alg:ALV-opt} and the PAPC algorithm, 
for $\chi=10^5$ and $\kappa=10^4$; other values gave similar plots. The computation time is roughly the same as the number of calls to $\mL$ and $\mL^T$ here.

Both algorithms converge linearly, but Algorithm~\ref{alg:ALV-opt} has a much better rate, which corresponds visually to the slope of the curves  in Figure~\ref{fig1}. 
Algorithm~\ref{alg:ALV-opt} makes $N = 317$ calls to $\mL$ and $\mL^T$ corresponding to the Chebyshev ‘inner loop’ between two gradient evaluations. It needs less
gradient calls than PAPC to achieve the same accuracy. The red curve in the right plot is the same as in the left plot, but stretched horizontally by a factor $N = 317$ (note the change in horizontal scale).

\section*{Acknowledgements}

Adil Salim was supported by KAUST and by a Simons-Berkeley Research Fellowship.

\bibliographystyle{abbrvnat}
\bibliography{ref.bib,biblio.bib}

\clearpage
\appendix
\part*{Appendix}

\section{Proof of Proposition~\ref{th:ALV}}
We denote by $\|\cdot\|_{\mQ}$ (resp. $\Dotprod{\cdot,\cdot}_{\mQ}$) the norm (resp. inner product) induced by $\mQ$, defined in~\eqref{eq:P}. The norm $\|\cdot\|_{\mQ}$ satisfies the following properties, stated as lemmas. 

\subsection{Preliminary Lemmas}
\begin{lemma}
	If the parameters $\eta > 0$ and $\theta > 0$ satisfy
	\begin{equation}\label{ALV:eq:etatheta}
		\eta\theta\lambda_{\max}(\mW) \leq 1,
	\end{equation}
	and if $\alpha > 0$, then the symmetric matrix $\mQ$ is positive definite and for every $x \in \sX$, $y \in \sY$, the following inequality holds:
	\begin{equation}
	\label{ALV:eq:Pbound}
		\frac{1}{\eta}\sqn{x} \leq \frac{1}{\eta}\sqn{x} + \frac{\eta\alpha}{\theta(1+\eta\alpha)}\|y\|^2 \leq \sqN{\begin{bmatrix}x\\y\end{bmatrix}}_\mQ \leq \frac{1}{\eta}\sqn{x} + \frac{1}{\theta}\sqn{y}.
	\end{equation}
\end{lemma}

\begin{proof}
The nonzero eigenvalues of $\mW = \mL^T \mL$ are the nonzero eigenvalues of $\mL \mL^T$, therefore $\lambda_{\max}(\mW) = \lambda_{\max}(\mL \mL^T)$. Consequently, using~\eqref{ALV:eq:etatheta},
\begin{equation*}
\frac{\eta}{1+\eta\alpha}\sqn{\mL^T y} \leq \frac{\eta}{1+\eta\alpha}\lambda_{\max}(\mW)\sqn{y} \leq \frac{\sqn{y}}{\theta(1+\eta\alpha)}.
\end{equation*}
    Therefore, since $\alpha \eta > 0$,
    \begin{equation*}
        \frac{1}{\eta}\sqn{x} + \left(\frac{1}{\theta} - \frac{1}{\theta(1+\eta\alpha)} \right)\sqn{y} \leq  \sqN{\begin{bmatrix}x\\y\end{bmatrix}}_\mQ = \frac{1}{\eta}\sqn{x} + \frac{1}{\theta}\sqn{y} - \frac{\eta}{1+\eta\alpha}\sqn{\mL^T y},
    \end{equation*}
    which proves in particular that $\mQ$ is positive definite.
\end{proof}

Besides, lines 5 to 7 of Algorithm~\ref{alg:ALV} admit the following representation, which is at the core of the convergence proof.
\begin{lemma}
The following equality holds:
	\begin{equation}\label{ALV:eq:xy}
		\mQ \begin{bmatrix}
		x^{k+1} - x^k\\y^{k+1} - y^k
		\end{bmatrix}
		=
		\begin{bmatrix}
			\alpha (x_g^k - x^{k+1}) - (\nabla F(x_g^k) + \mL^T y^{k+1})\\
			\mL x^{k+1} - b
		\end{bmatrix}.
	\end{equation}
    
\end{lemma}

\begin{proof}
	Using the definition of $\mQ$, we have
	\begin{align*}
	\mQ \begin{bmatrix}
	x^{k+1} - x^k\\y^{k+1} - y^k
	\end{bmatrix}
	=
	\begin{bmatrix}
	\frac{1}{\eta}(x^{k+1} - x^k)\\
	\frac{1}{\theta}(y^{k+1} - y^k) - \frac{\eta}{1+\eta\alpha}\mL \mL^T(y^{k+1} - y^k)
	\end{bmatrix}.
	\end{align*}
	\pagebreak
	From line 7 of Algorithm~\ref{alg:ALV} it follows that
	\begin{align*}
		\frac{1}{\eta}(x^{k+1} - x^k) = \alpha(x_g^k - x^{k+1})-(\nabla F(x_g^k) + \mL^T y^{k+1}),
	\end{align*}
	and from line 6 of Algorithm~\ref{alg:ALV}
	\begin{align*}
		y^{k+1} - y^k = \theta (\mL x^{k+\frac{1}{2}} - b).
	\end{align*}
	Hence,
	\begin{align*}
	\mQ \begin{bmatrix}
	x^{k+1} - x^k\\y^{k+1} - y^k
	\end{bmatrix}
	=
	\begin{bmatrix}
	\alpha(x_g^k - x^{k+1})-(\nabla F(x_g^k) + \mL^T y^{k+1})\\
	(\mL x^{k+\frac{1}{2}} - b) - \frac{\eta}{1+\eta\alpha}\mL \mL^T(y^{k+1} - y^k)
	\end{bmatrix}.
	\end{align*}
	From lines 5 and 7 of Algorithm~\ref{alg:ALV},
	\begin{equation}
		x^{k+1} - x^{k+\frac{1}{2}} = \frac{-\eta}{1+\eta\alpha} \mL^T(y^{k+1} - y^k),
	\end{equation}
	therefore,
	\begin{equation*}
	    (\mL x^{k+\frac{1}{2}} - b) - \frac{\eta}{1+\eta\alpha}\mL \mL^T(y^{k+1} - y^k) = \mL\left(x^{k+\frac{1}{2}}- \frac{\eta}{1+\eta\alpha} \mL^T(y^{k+1} - y^k) \right) - b = \mL x^{k+1} - b.
	\end{equation*}
	Finally, 
	\begin{align*}
	\mQ \begin{bmatrix}
	x^{k+1} - x^k\\y^{k+1} - y^k
	\end{bmatrix}
	=
	\begin{bmatrix}
	\alpha(x_g^k - x^{k+1})-(\nabla F(x_g^k) + \mL^T y^{k+1})\\
	\mL x^{k+1} - b
	\end{bmatrix}.
	\end{align*}
\end{proof}

We now start the proof of Proposition~\ref{th:ALV}.

\begin{lemma}\label{ALV:lem:1}
	Suppose that $\alpha$ satisfies $0\leq \alpha \leq \mu$.
	Then the following inequality holds:
	\begin{align}\label{ALV:eq:1}
		-\frac{1}{2\eta}\sqn{x^{k+1} - x^k} \leq -\frac{\eta}{4}\sqn{\mL^T y^{k+1} - \mL^T y^{\star}}
		+
		\eta\alpha^2\sqn{x^{k+1} - x^{\star}}
		+
		2\eta L\bg_{F}(x_g^k, x^{\star}).
	\end{align}
\end{lemma}
\begin{proof}
	From line 7 of Algorithm~\ref{alg:ALV} and the optimality condition 
	$\nabla F(x^{\star}) + \mL^T y^{\star} = 0$, it follows that
	\begin{align*}
		\sqn{x^{k+1} - x^k}
		&=
		\sqn{\eta(\mL^T y^{k+1} - \mL^T y^{\star}) + \eta(\nabla F(x_g^k) - \nabla F(x^{\star}) - \alpha (x_g^k - x^{\star})) + \eta\alpha(x^{k+1} - x^{\star})}
		\\&\geq
		\frac{\eta^2}{2}\sqn{\mL^T y^{k+1} - \mL^T y^{\star}}
		-
		2\eta^2\alpha^2\sqn{x^{k+1} - x^{\star}}
		\\&-
		2\eta^2\sqn{\nabla F(x_g^k) - \nabla F(x^{\star}) - \alpha (x_g^k - x^{\star})},
	\end{align*}
	where we used $\|a+b+c\|^2 \geq 0.5\|a\|^2 - 2\|b\|^2 - 2\|c\|^2$. 
	Let $\bar{F}(x) \eqdef F(x)  - \frac{\alpha}{2}\sqn{x}$. The function $\bar{F}$ is a convex and $(L-\alpha)$-smooth function, therefore $\sqn{\nabla \bar{F}(x) - \nabla \bar{F}(x')} \leq 2(L-\alpha)\bg_{\bar{F}}(x,x')$. Therefore, we can lower bound the last term and get
	\begin{align*}
	\sqn{x^{k+1} - x^k}
	&\geq
	\frac{\eta^2}{2}\sqn{\mL^T y^{k+1} - \mL^T y^{\star}}
	-
	2\eta^2\alpha^2\sqn{x^{k+1} - x^{\star}}
	-
	4\eta^2(L - \alpha)\bg_{\bar{F}}(x_g^k, x^{\star})
	\\&\geq
	\frac{\eta^2}{2}\sqn{\mL^T y^{k+1} - \mL^T y^{\star}}
	-
	2\eta^2\alpha^2\sqn{x^{k+1} - x^{\star}}
	-
	4\eta^2L\bg_{F}(x_g^k, x^{\star}).
	\end{align*}
	Rearranging and dividing by $2\eta$ concludes the proof.
\end{proof}

Our last lemma states the linear convergence of a Lyapunov function to zero.

\begin{lemma}\label{ALV:lem:2}
Consider $\lambda_1 \geq \lambda_{\max}(\mW)$ and $\lambda_2 \leq \lambda_{\min}^+(\mW)$.

	Let parameter $\eta$ be defined by
	\begin{equation}\label{ALV:eq:eta}
		\eta = \frac{1} {4\tau L}.
	\end{equation}
	Let us set the parameter $\theta$ as
	\begin{equation}\label{ALV:eq:theta}
		\theta = \frac{1}{\eta\lambda_1}.
	\end{equation}
	Let us set the parameter $\alpha$ as
	\begin{equation}\label{ALV:eq:alpha}
		\alpha = \mu.
	\end{equation}
	Let us set the parameter $\tau$ as
	\begin{equation}\label{ALV:eq:tau}
		\tau = \min \left\{1,  \frac{1}{2}\sqrt{\frac{\mu}{L}\frac{\lambda_1}{\lambda_2}}\right\}.
	\end{equation}
	Let $\Psi^k$ be the following Lyapunov function:
	\begin{equation}\label{ALV:eq:Psi}
		\Psi^k = \sqN{\begin{bmatrix}
			x^{k} - x^{\star}\\y^{k} - y^{\star}
			\end{bmatrix}}_\mQ
		+
		\frac{2(1-\tau)}{\tau}\bg_F(x_f^k, x^{\star}),
	\end{equation}
	Then the following inequality holds:
	\begin{equation}
		\Psi^{k+1}
		\leq
		\left(1 + \frac{1}{4}\min\left\{\sqrt{\frac{\mu}{L}\frac{\lambda_2}{\lambda_1}},\frac{\lambda_2}{\lambda_1}\right\}\right)^{-1}
		\Psi^k.
	\end{equation}
\end{lemma}

\begin{proof}
	\begin{align*}
		\sqN{\begin{bmatrix}
				x^{k+1} - x^{\star}\\y^{k+1} - y^{\star}
			\end{bmatrix}}_\mQ
		&=
		\sqN{\begin{bmatrix}
			x^{k} - x^{\star}\\y^{k} - y^{\star}
			\end{bmatrix}}_\mQ
		-
		\sqN{\begin{bmatrix}
			x^{k+1} - x^k\\y^{k+1} - y^k
			\end{bmatrix}}_\mQ
		+
		2\Dotprod{\begin{bmatrix}
			x^{k+1} - x^k\\y^{k+1} - y^k
			\end{bmatrix}, \begin{bmatrix}
			x^{k+1} - x^{\star}\\y^{k+1} - y^{\star}
			\end{bmatrix}}_{\mQ}.
	\end{align*}
	Note that stepsize $\eta$ defined by \eqref{ALV:eq:eta} and stepsize $\theta$ defined by \eqref{ALV:eq:theta} satisfy \eqref{ALV:eq:etatheta}, hence inequality \eqref{ALV:eq:Pbound} holds.
	Using \eqref{ALV:eq:Pbound} and \eqref{ALV:eq:xy} we get
	\begin{align*}
	\sqN{\begin{bmatrix}
		x^{k+1} - x^{\star}\\y^{k+1} - y^{\star}
		\end{bmatrix}}_\mQ
	&\leq
	\sqN{\begin{bmatrix}
		x^{k} - x^{\star}\\y^{k} - y^{\star}
		\end{bmatrix}}_\mQ
	-
	\frac{1}{\eta}\sqn{x^{k+1} - x^k}
	\\&\quad+
	2\Dotprod{\begin{bmatrix}
			\alpha (x_g^k - x^{k+1}) - (\nabla F(x_g^k) + \mL^T y^{k+1})\\
			\mL x^{k+1} - b
		\end{bmatrix}, \begin{bmatrix}
		x^{k+1} - x^{\star}\\y^{k+1} - y^{\star}
		\end{bmatrix}}\\
	&=
	\sqN{\begin{bmatrix}
		x^{k} - x^{\star}\\y^{k} - y^{\star}
		\end{bmatrix}}_\mQ
	-
	\frac{1}{\eta}\sqn{x^{k+1} - x^k} +
	2\alpha\<x_g^k - x^{k+1}, x^{k+1} - x^{\star}>
	\\
	\\&
	\quad-2\Dotprod{\begin{bmatrix}
			\nabla F(x_g^k) + \mL^T y^{k+1}\\
			-\mL x^{k+1} + b
		\end{bmatrix} - \begin{bmatrix}\nabla F(x^{\star}) + \mL^T y^{\star}\\
			-\mL x^{\star} + b
		\end{bmatrix}, \begin{bmatrix}
		x^{k+1} - x^{\star}\\y^{k+1} - y^{\star}
		\end{bmatrix}},
	\end{align*}
	using the primal--dual optimality conditions~\eqref{eq:pd}. Rewrite 
	\begin{equation*}
	    \begin{bmatrix}
			\nabla F(x_g^k) + \mL^T y^{k+1}\\
			-\mL x^{k+1} + b
		\end{bmatrix} - \begin{bmatrix}\nabla F(x^{\star}) + \mL^T y^{\star}\\
			-\mL x^{\star} + b
		\end{bmatrix} = \begin{bmatrix}
			\nabla F(x_g^k) - \nabla F(x^{\star})\\
			b - b
		\end{bmatrix} + \begin{bmatrix} 0 & \mL^T \\
			-\mL & 0 
		\end{bmatrix} \begin{bmatrix}
			x^{k+1} - x^{\star}\\
			y^{k+1} - y^{\star}
		\end{bmatrix}.
	\end{equation*}
	Using $\Dotprod{A z,z} = 0$ for any skew-symmetric matrix $A$, we obtain
	\begin{align*}
	    -2\Dotprod{\begin{bmatrix}
			\nabla F(x_g^k) + \mL^T y^{k+1}\\
			-\mL x^{k+1} + b
		\end{bmatrix} - \begin{bmatrix}\nabla F(x^{\star}) + \mL^T y^{\star}\\
			-\mL x^{\star} + b
		\end{bmatrix}, \begin{bmatrix}
		x^{k+1} - x^{\star}\\y^{k+1} - y^{\star}
		\end{bmatrix}}\\
		 = -
	2\<\nabla F(x_g^k) - \nabla F(x^{\star}), x^{k+1} - x^{\star}>.
	\end{align*}
	Hence,
	\begin{align*}
	\sqN{\begin{bmatrix}
		x^{k+1} - x^{\star}\\y^{k+1} - y^{\star}
		\end{bmatrix}}_\mQ
	&\leq
	\sqN{\begin{bmatrix}
		x^{k} - x^{\star}\\y^{k} - y^{\star}
		\end{bmatrix}}_\mQ
	-
	\frac{1}{\eta}\sqn{x^{k+1} - x^k}
	+
	2\alpha\<x_g^k - x^{k+1}, x^{k+1} - x^{\star}>
	\\&\quad-
	2\<\nabla F(x_g^k) - \nabla F(x^{\star}), x^{k+1} - x^{\star}>
	\\&=
	\sqN{\begin{bmatrix}
		x^{k} - x^{\star}\\y^{k} - y^{\star}
		\end{bmatrix}}_\mQ
	-
	\frac{1}{\eta}\sqn{x^{k+1} - x^k}
	-
	2\alpha\sqn{x^{k+1} - x^{\star}}
	\\&\quad-
	2\alpha\<x_g^k -  x^{\star}, x^{k+1} - x^{\star}>
	-
	2\<\nabla F(x_g^k) - \nabla F(x^{\star}), x^{k+1} - x^{\star}>.
	\end{align*}
	Using Young's  inequality $2\<a,b> \leq \sqn{a} + \sqn{b}$ we get
	\begin{align*}
	\sqN{\begin{bmatrix}
		x^{k+1} - x^{\star}\\y^{k+1} - y^{\star}
		\end{bmatrix}}_\mQ
	&\leq
	\sqN{\begin{bmatrix}
		x^{k} - x^{\star}\\y^{k} - y^{\star}
		\end{bmatrix}}_\mQ
	-
	\frac{1}{\eta}\sqn{x^{k+1} - x^k}
	-
	2\alpha\sqn{x^{k+1} - x^{\star}}
	\\&\quad+
	\alpha\sqn{x_g^k - x^{\star}}
	+
	\alpha\sqn{x^{k+1} - x^{\star}}
	-
	2\<\nabla F(x_g^k) - \nabla F(x^{\star}), x^{k+1} - x^{\star}>
	\\&=
	\sqN{\begin{bmatrix}
		x^{k} - x^{\star}\\y^{k} - y^{\star}
		\end{bmatrix}}_\mQ
	-
	\frac{1}{\eta}\sqn{x^{k+1} - x^k}
	-
	\alpha\sqn{x^{k+1} - x^{\star}}
	+
	\alpha\sqn{x_g^k - x^{\star}}
	\\&\quad-
	2\<\nabla F(x_g^k) - \nabla F(x^{\star}), x^{k+1} - x^{\star}>.
	\end{align*}
	Using line 4 of Algorithm~\ref{alg:ALV}, we have $x^k - x^{\star} = (x_g^k - x^{\star}) + \frac{1-\tau}{\tau}(x_g^k - x_f^k)$ and using line 8, $x^{k+1} - x^k = \frac{2-\tau}{2\tau}(x_f^{k+1} - x_g^{k})$. Therefore, decomposing $x^{k+1} - x^{\star} = (x^{k+1} - x^k) + (x^k - x^{\star})$,
	\begin{align*}
	\sqN{\begin{bmatrix}
		x^{k+1} - x^{\star}\\y^{k+1} - y^{\star}
		\end{bmatrix}}_\mQ
	&\leq
	\sqN{\begin{bmatrix}
		x^{k} - x^{\star}\\y^{k} - y^{\star}
		\end{bmatrix}}_\mQ
	-
	\alpha\sqn{x^{k+1} - x^{\star}}
	+
	\alpha\sqn{x_g^k - x^{\star}}
	-
	\frac{1}{2\eta}\sqn{x^{k+1} - x^k}
	\\&-
	\frac{2-\tau}{\tau}
	\left(
			\<\nabla F(x_g^k) - \nabla F(x^{\star}), x_f^{k+1} - x_g^k>
			+
			\frac{1}{2\eta}\frac{(2-\tau)}{4\tau}\sqn{x_f^{k+1} - x_g^k}
	\right)
	\\&-
	2\<\nabla F(x_g^k) - \nabla F(x^{\star}), x_g^k - x^{\star}>
	+
	\frac{2(1-\tau)}{\tau}\<\nabla F(x_g^k) - \nabla F(x^{\star}), x_f^k - x_g^k>.
	\end{align*}
	Since  $\eta$ defined by \eqref{ALV:eq:eta} satisfies $\eta \leq \frac{2-\tau}{4\tau L}$, we get
	\begin{align*}
	\sqN{\begin{bmatrix}
		x^{k+1} - x^{\star}\\y^{k+1} - y^{\star}
		\end{bmatrix}}_\mQ
	&\leq
	\sqN{\begin{bmatrix}
		x^{k} - x^{\star}\\y^{k} - y^{\star}
		\end{bmatrix}}_\mQ
	-
	\alpha\sqn{x^{k+1} - x^{\star}}
	+
	\alpha\sqn{x_g^k - x^{\star}}
	-
	\frac{1}{2\eta}\sqn{x^{k+1} - x^k}
	\\&\quad-
	\frac{2-\tau}{\tau}
	\left(
	\<\nabla F(x_g^k) - \nabla F(x^{\star}), x_f^{k+1} - x_g^k>
	+
	\frac{L}{2}\sqn{x_f^{k+1} - x_g^k}
	\right)
	\\&\quad-
	2\<\nabla F(x_g^k) - \nabla F(x^{\star}), x_g^k - x^{\star}>\\
	&\quad+
	\frac{2(1-\tau)}{\tau}\<\nabla F(x_g^k) - \nabla F(x^{\star}), x_f^k - x_g^k>.
	\end{align*}
	Using $\mu$-strong convexity and $L$-smoothness of $F$ we get
	\begin{align*}
	\sqN{\begin{bmatrix}
		x^{k+1} - x^{\star}\\y^{k+1} - y^{\star}
		\end{bmatrix}}_\mQ
	&\leq
	\sqN{\begin{bmatrix}
		x^{k} - x^{\star}\\y^{k} - y^{\star}
		\end{bmatrix}}_\mQ
	-
	\alpha\sqn{x^{k+1} - x^{\star}}
	+
	\alpha\sqn{x_g^k - x^{\star}}
	-
	\frac{1}{2\eta}\sqn{x^{k+1} - x^k}
	\\&\quad-
	\frac{2-\tau}{\tau}
	\left(
		\bg_F (x_f^{k+1},x^{\star}) - \bg_F (x_g^k,x^{\star})
	\right)
	\\&\quad+
	\frac{2(1-\tau)}{\tau}\left( \bg_F(x_f^k, x^{\star}) - \bg_F(x_g^k,x^{\star}) \right)
	\\&\quad-
	2\left(\bg_F(x_g^k,x^{\star}) + \frac{\mu}{2}\sqn{x_g^k - x^{\star}}\right)
	\\&=
	\sqN{\begin{bmatrix}
		x^{k} - x^{\star}\\y^{k} - y^{\star}
		\end{bmatrix}}_\mQ
	-
	\alpha\sqn{x^{k+1} - x^{\star}}
	+
	\frac{2(1-\tau)}{\tau}\bg_F(x_f^k, x^{\star})
	\\&\quad-
	\frac{2-\tau}{\tau}\bg_F (x_f^{k+1},x^{\star}) 
	\\&\quad+
	(\alpha - \mu)\sqn{x_g^k - x^{\star}}
	-
	\frac{1}{2\eta}\sqn{x^{k+1} - x^k}
	-
	\bg_F(x_g^k,x^{\star}).
	\end{align*}
	Now, we define $\delta = \min \left\{1, \frac{1}{2\eta L}\right\}$. Since $\alpha$ defined by \eqref{ALV:eq:alpha} satisfies conditions of Lemma~\ref{ALV:lem:1}, we can use \eqref{ALV:eq:1} and get
	\begin{align*}
	\sqN{\begin{bmatrix}
		x^{k+1} - x^{\star}\\y^{k+1} - y^{\star}
		\end{bmatrix}}_\mQ
	&\leq
	\sqN{\begin{bmatrix}
		x^{k} - x^{\star}\\y^{k} - y^{\star}
		\end{bmatrix}}_\mQ
	-
	\alpha\sqn{x^{k+1} - x^{\star}}
	+
	\frac{2(1-\tau)}{\tau}\bg_F(x_f^k, x^{\star})\\
	&\quad-
	\frac{2-\tau}{\tau}\bg_F (x_f^{k+1},x^{\star}) 
	\\&\quad+
	(\alpha - \mu)\sqn{x_g^k - x^{\star}}
	-
	\frac{\delta}{2\eta}\sqn{x^{k+1} - x^k}
	-
	\bg_F(x_g^k,x^{\star})
	\\&\leq
	\sqN{\begin{bmatrix}
		x^{k} - x^{\star}\\y^{k} - y^{\star}
		\end{bmatrix}}_\mQ
	-
	\alpha\sqn{x^{k+1} - x^{\star}}
	+
	\frac{2(1-\tau)}{\tau}\bg_F(x_f^k, x^{\star})
	\\&\quad-
	\frac{2-\tau}{\tau}\bg_F (x_f^{k+1},x^{\star}) 
	-
	\frac{\eta\delta}{4}\sqn{\mL^T y^{k+1} - \mL^T y^{\star}}
	+
	\eta\alpha^2\delta\sqn{x^{k+1} - x^{\star}}\\
	&\quad+
	2\eta  L \delta\bg_{f }(x_g^k, x^{\star})
	+
	(\alpha - \mu)\sqn{x_g^k - x^{\star}}
	-
	\bg_F(x_g^k,x^{\star})
	\\&\leq
	\sqN{\begin{bmatrix}
		x^{k} - x^{\star}\\y^{k} - y^{\star}
		\end{bmatrix}}_\mQ
	-
	\alpha\sqn{x^{k+1} - x^{\star}}
	+
	\frac{2(1-\tau)}{\tau}\bg_F(x_f^k, x^{\star})
	\\&\quad-
	\frac{2-\tau}{\tau}\bg_F (x_f^{k+1},x^{\star}) 
	-
	\frac{\eta\delta}{4}\sqn{\mL^T y^{k+1} - \mL^T y^{\star}}
	+
	\frac{\alpha^2}{2L}\sqn{x^{k+1} - x^{\star}}
	\\&\quad+
	(\alpha - \mu)\sqn{x_g^k - x^{\star}}
	\\&=
	\sqN{\begin{bmatrix}
		x^{k} - x^{\star}\\y^{k} - y^{\star}
		\end{bmatrix}}_\mQ
	-
	\left(
		\alpha - \frac{\alpha^2}{2L}
	\right)\sqn{x^{k+1} - x^{\star}}
	-
	\frac{\eta\delta}{4}\sqn{\mL^T y^{k+1} - \mL^T y^{\star}}
	\\&\quad+
	\frac{2(1-\tau)}{\tau}\bg_F(x_f^k, x^{\star})
	-
	\frac{2-\tau}{\tau}\bg_F (x_f^{k+1},x^{\star}) 
	+
	(\alpha - \mu)\sqn{x_g^k - x^{\star}}.
	\end{align*}
	Using the parameter $\alpha = \mu$ defined in \eqref{ALV:eq:alpha} and using $\mu \leq L$, we get
	\begin{align*}
	\sqN{\begin{bmatrix}
		x^{k+1} - x^{\star}\\y^{k+1} - y^{\star}
		\end{bmatrix}}_\mQ
	&\leq
	\sqN{\begin{bmatrix}
		x^{k} - x^{\star}\\y^{k} - y^{\star}
		\end{bmatrix}}_\mQ
	-
	\frac{\mu}{2}\sqn{x^{k+1} - x^{\star}}
	-
	\frac{\eta\delta}{4}\sqn{\mL^T y^{k+1} - \mL^T y^{\star}}
	\\&\quad+
	\frac{2(1-\tau)}{\tau}\bg_F(x_f^k, x^{\star})
	-
	\frac{2-\tau}{\tau}\bg_F (x_f^{k+1},x^{\star}) .
	\end{align*}
	For every $y \in \range(\mL)$, $\lambda_2\|y\|^2 \leq \lambda_{\min}^{+}(\mW)\|y\|^2 \leq \|\mL^T y\|^2$. Using line 6 of Algorithm~\ref{alg:ALV}, one can check by induction that $y^k \in \range(\mL)$ for every $k \geq 0$. Moreover, using~\eqref{eq:pd}, $y^{\star}\in \range(\mL)$. Therefore, 
	\begin{align*}
	\sqN{\begin{bmatrix}
		x^{k+1} - x^{\star}\\y^{k+1} - y^{\star}
		\end{bmatrix}}_\mQ
	&\leq
	\sqN{\begin{bmatrix}
		x^{k} - x^{\star}\\y^{k} - y^{\star}
		\end{bmatrix}}_\mQ
	-
	\frac{\mu}{2}\sqn{x^{k+1} - x^{\star}}
	-
	\frac{\eta\delta\lambda_2}{4}\sqn{y^{k+1} - y^{\star}}
	\\&\quad+
	\frac{2(1-\tau)}{\tau}\bg_F(x_f^k, x^{\star})
	-
	\frac{2-\tau}{\tau}\bg_F (x_f^{k+1},x^{\star}).
	\end{align*}
	Using \eqref{ALV:eq:Pbound} we get
	\begin{align*}
	\sqN{\begin{bmatrix}
		x^{k+1} - x^{\star}\\y^{k+1} - y^{\star}
		\end{bmatrix}}_\mQ
	&\leq
	\sqN{\begin{bmatrix}
		x^{k} - x^{\star}\\y^{k} - y^{\star}
		\end{bmatrix}}_\mQ
	-
	\min\left\{\frac{\eta\mu}{2}, \frac{\eta\theta\delta\lambda_2}{4}\right\}
	\sqN{\begin{bmatrix}
		x^{k+1} - x^{\star}\\y^{k+1} - y^{\star}
		\end{bmatrix}}_\mQ
	\\&\quad+
	\frac{2(1-\tau)}{\tau}\bg_F(x_f^k, x^{\star})
	-
	\frac{2-\tau}{\tau}\bg_F (x_f^{k+1},x^{\star}) .
	\end{align*}
	Using the parameter $\theta$ defined in \eqref{ALV:eq:theta} and the definition of $\delta$, we get
	\begin{align*}
	\sqN{\begin{bmatrix}
		x^{k+1} - x^{\star}\\y^{k+1} - y^{\star}
		\end{bmatrix}}_\mQ
	&\leq
	\sqN{\begin{bmatrix}
	x^{k} - x^{\star}\\y^{k} - y^{\star}
	\end{bmatrix}}_\mQ
	-
	 \min\left\{\frac{\eta\mu}{2},\frac{\lambda_2}{4\lambda_1}, \frac{\lambda_2}{8\eta L\lambda_1} \right\}
	\sqN{\begin{bmatrix}
		x^{k+1} - x^{\star}\\y^{k+1} - y^{\star}
		\end{bmatrix}}_\mQ
	\\&\quad+
	\frac{2(1-\tau)}{\tau}\bg_F(x_f^k, x^{\star})
	-
	\frac{2-\tau}{\tau}\bg_F (x_f^{k+1},x^{\star}).
	\end{align*}
	Plugging the parameter $\eta$ defined in \eqref{ALV:eq:eta}, we get
	\begin{align*}
	\sqN{\begin{bmatrix}
		x^{k+1} - x^{\star}\\y^{k+1} - y^{\star}
		\end{bmatrix}}_\mQ
	&\leq
	\sqN{\begin{bmatrix}
		x^{k} - x^{\star}\\y^{k} - y^{\star}
		\end{bmatrix}}_\mQ
	-
	\min\left\{\frac{\mu}{8\tau L},
	\frac{\lambda_2}{4\lambda_1}, \frac{\tau\lambda_2}{2\lambda_1} \right\}
	\sqN{\begin{bmatrix}
		x^{k+1} - x^{\star}\\y^{k+1} - y^{\star}
		\end{bmatrix}}_\mQ
	\\&\quad+
	\frac{2(1-\tau)}{\tau}\bg_F(x_f^k, x^{\star})
	-
	\frac{2-\tau}{\tau}\bg_F (x_f^{k+1},x^{\star})
	\\&\leq
	\sqN{\begin{bmatrix}
		x^{k} - x^{\star}\\y^{k} - y^{\star}
		\end{bmatrix}}_\mQ
	-
	\min\left\{\frac{\mu}{8\tau L},
	\frac{\lambda_2}{4\lambda_1}, \frac{\tau\lambda_2}{2\lambda_1} \right\}
	\sqN{\begin{bmatrix}
		x^{k+1} - x^{\star}\\y^{k+1} - y^{\star}
		\end{bmatrix}}_\mQ
	\\&\quad+
	\frac{2(1-\tau)}{\tau}\bg_F(x_f^k, x^{\star})
	-
	\left(1 + \frac{\tau}{2}\right)\frac{2(1-\tau)}{\tau}\bg_F (x_f^{k+1},x^{\star}).
	\end{align*}
	After rearranging the terms and using the definition of $\Psi^k$ in \eqref{ALV:eq:Psi}, we get
	\begin{align*}
		\Psi^k
		&\geq
		\left(
		1 + \min\left\{\frac{\tau}{2},\frac{\mu}{8\tau L},
		\frac{\lambda_2}{4\lambda_1}, \frac{\tau\lambda_2}{2\lambda_1} \right\}
		\right) \Psi^{k+1}.
	\end{align*}
	Plugging the parameter $\tau$ defined in \eqref{ALV:eq:tau}, we get
	\begin{align*}
	\Psi^{k}
	&\geq
	\left(1 + \frac{1}{4}\min\left\{\sqrt{\frac{\mu}{L}\frac{\lambda_2}{\lambda_1}},\frac{\lambda_2}{\lambda_1}\right\}\right)
	\Psi^{k+1}.
	\end{align*}
\end{proof}

\subsection{End of the Proof of Proposition~\ref{th:ALV}}
	The conditions of Lemma~\ref{ALV:lem:2} are satisfied, hence the following inequality holds for every $k\geq 0$:
	\begin{equation*}
		\Psi^{k+1}
		\leq
		\left(1 + \frac{1}{4}\min\left\{\sqrt{\frac{\mu}{L}\frac{\lambda_2}{\lambda_1}},\frac{\lambda_2}{\lambda_1}\right\}\right)^{-1}
		\Psi^k.
	\end{equation*}
	After telescoping we get
	\begin{equation*}
		\Psi^k \leq \left(1 + \frac{1}{4}\min\left\{\sqrt{\frac{\mu}{L}\frac{\lambda_2}{\lambda_1}},\frac{\lambda_2}{\lambda_1}\right\}\right)^{-k}
		\Psi^0.
	\end{equation*}
	Inequality \eqref{ALV:eq:Pbound} implies $\Psi^0 \leq C$, where $C \eqdef \frac{1}{\eta}\sqN{x^0 - x^{\star}} + \frac{1}{\theta}\sqn{y^0 - y^{\star}} +
		\frac{2(1-\tau)}{\tau}\bg_F(x_f^0, x^{\star}).$
	Hence, we obtain
	\begin{equation}
	\label{eq:last}
	\Psi^k \leq \left(1 + \frac{1}{4}\min\left\{\sqrt{\frac{\mu}{L}\frac{\lambda_2}{\lambda_1}},\frac{\lambda_2}{\lambda_1}\right\}\right)^{-k}
	C.
	\end{equation}
	It remains to lower bound $\Psi^k$ using \eqref{ALV:eq:Pbound} one more time:
	\begin{equation*}
		\frac{1}{\eta}\sqN{x^k - x^{\star}} + \frac{\eta\alpha}{\theta(1+\eta\alpha)}\sqN{y^k - y^{\star}}
		+\frac{2(1-\tau)}{\tau}\bg_F(x_f^k, x^{\star})
		\leq
		\Psi^k.
	\end{equation*}
	Combining with~\eqref{eq:last} gives the result.\hfill$\square$

\section{Proof of Theorem~\ref{th:main}}

\subsection{Proof of Equation~\eqref{eqch11}}\label{secab1}

The vector $z^N = \mathrm{Chebyshev}(x,\mL,b,N)$ is the $N^{\text{th}}$ iterate of the Chebyshev iteration, which amounts to  applying the Chebyshev polynomials $\widetilde{\boldsymbol{T}}_N$ to the residual $\mL^T (\mL z^{0}-b)$, to make it converge to zero. So,  $z^N$ satisfies 
\begin{equation}
\mL^T (\mL z^{N}-b)=\widetilde{\boldsymbol{T}}_N(\mW)\big(\mL^T (\mL z^{0}-b)\big),
\end{equation}
so that $\|\mL z^{N}-b\|$ converges linearly to zero when $N\rightarrow +\infty$.

Since $\widetilde{\boldsymbol{T}}_N(0) = 1$, there exists a polynomial $\widetilde{\boldsymbol{R}}_N$ such that $\widetilde{\boldsymbol{T}}_N(X)=1+X\widetilde{\boldsymbol{R}}_N(X)$. Therefore,
\begin{equation*}
\mL^T (\mL z^{n}-b)=\big(\mL^T (\mL z^{0}-b)\big) + \mW \widetilde{\boldsymbol{R}}_N(\mW)\big(\mL^T (\mL z^{0}-b)\big);
\end{equation*}
that is,
\begin{equation*}
\mW z^N = \mW \left(z^0 + \widetilde{\boldsymbol{R}}_N(\mW)\big(\mL^T (\mL z^{0}-b)\big)\right).
\end{equation*}
One can check by induction that $z^N \in z^0+ \range(\mW)$. Using $\mI_\sX + \mW \widetilde{\boldsymbol{R}}_N(\mW) = \widetilde{\boldsymbol{T}}_N(\mW)$,
\begin{align*}
 z^{n}&=z^0+\widetilde{\boldsymbol{R}}_N(\mW)\big(\mL^T (\mL z^{0}-b)\big)\\
 &=z^0+\mW \widetilde{\boldsymbol{R}}_N(\mW) z^0 - \widetilde{\boldsymbol{R}}_N(\mW)\mL^T b\\
 &=\widetilde{\boldsymbol{T}}_N(\mW) z^{0}-\widetilde{\boldsymbol{R}}_N(\mW)\mL^T b\\
 &=\widetilde{\boldsymbol{T}}_N(\mW) z^{0}-\widetilde{\boldsymbol{R}}_N(\mW)\mW x^\star\\
  &=\widetilde{\boldsymbol{T}}_N(\mW) z^{0}-\widetilde{\boldsymbol{T}}_N(\mW) x^\star + x^\star.
\end{align*}
 Finally, for every $z^0\in \sX$,
\begin{equation*}
{\mP(\mW)}z^0- {\mP(\mW)} x^{\star} = z^0 - \widetilde{\boldsymbol{T}}_N(\mW)z^0 - x^\star+\widetilde{\boldsymbol{T}}_N(\mW)x^\star
=z^0 - z^N.
\end{equation*}

\subsection{End of proof of Theorem~\ref{th:main}}

In Sections~\ref{sec:algo-equiv-P} and \ref{secab1}, we proved that Algorithm~\ref{alg:ALV} applied to the equivalent Problem~\eqref{eq:equiv-pb} is equivalent to our main Algorithm~\ref{alg:ALV-opt}. Therefore, we can prove our main Theorem~\ref{th:main} by applying Proposition~\ref{th:ALV} to Problem~\eqref{eq:equiv-pb}. Indeed, the proof of Theorem~\ref{th:main} is a direct application of Proposition~\ref{th:ALV} to Problem~\eqref{eq:equiv-pb}, using that $N \geq \sqrt{\chi}$ implies $\lambda_{\max}(\mP(\mW)) \leq 19/15, \lambda_{\min}^+(\mP(\mW)) \geq 11/15$ and $\chi(\mP(\mW)) \leq 19/11$, see Section~\ref{sec:chiP}.\hfill $\square$

\end{document}